\documentclass{amsart}
\usepackage{amsmath,enumerate,enumitem,amsfonts,color,amsthm,hyperref,bbm} 

\usepackage{appendix}

\title[Discretization of the Martingale Representation Theorem]{Scaling Limits for the Discretization of the Martingale Representation Theorem}

\author{Yan Dolinksy} 
\address{Department of Statistics, Hebrew University}
\email{yan.dolinsky@mail.huji.ac.il}

\date{\today}

\numberwithin{equation}{section}  
\newtheorem{defn}{Definition}[section]

\newtheorem{remark}[defn]{Remark}
\newtheorem{theorem}[defn]{Theorem}

\begin{document}

\begin{abstract}
In this note, we derive a large-deviation-type scaling limit for a discretization of the Martingale Representation Theorem. Somewhat surprisingly, and to the best of our knowledge, this result has not been previously obtained in the literature.
\end{abstract}

\thanks{Partially supported by the ISF grant 305/25. }

\maketitle

\section{Introduction and the Main Result}\label{sec:1}
Let
$W_t$, $t\geq 0$, be a standard one-dimensional Brownian motion defined on a probability space $(\Omega,\mathcal F,\mathbb P)$ and endowed with its augmented natural filtration $(\mathcal F_t)_{t\geq 0}$.
The well-known Martingale Representation Theorem (see \cite[Section 4.7]{KS}) states that for a sufficiently regular function \(f:\mathbb R\to\mathbb R\),
\[
f(W_1)=u(0,0)+\int_{0}^{1} u_x(t,W_t)\,dW_t,
\]
where \(u=u(t,x)\) is the unique solution (under suitable growth conditions) of the heat equation
\[
u_t+\frac{1}{2}u_{xx}=0,
\qquad
u(1,x)=f(x),
\]
and is given by
\[
u(t,x):=\mathbb E\left[f(x+W_{1-t})\right]
=\frac{1}{\sqrt{2\pi}}
\int_{\mathbb R}
f(x+\sqrt{1-t}\,y)e^{-\frac{y^2}{2}}\,dy.
\]

From the point of view of mathematical finance, the Martingale Representation Theorem can be viewed as a perfect replication result, and the process
\[
\bigl(u_x(t,W_t)\bigr)_{t\in[0,1]}
\]
is the corresponding delta-hedging strategy.
In real market conditions, however, portfolio rebalancing can occur only at discrete times.
Roughly speaking, as the frequency of rebalancing increases and the discrete hedging strategy approaches continuous rebalancing, the corresponding portfolio values converge to those of the continuously rebalanced portfolio.

We now arrive at our main result, which will be proved in the next section.

\begin{theorem}\label{thm1}
Let \(f\in C^3(\mathbb{R})\) be such that \(f''\) is uniformly bounded,
\[
\sup_{x\in\mathbb{R}} f''(x) < 1,
\]
and \(f'''\) is Lipschitz continuous.

Set
\[
S_n
:=
f(W_1)-u(0,0)
-
\sum_{k=0}^{n-1}
u_x\!\left(\frac{k}{n},W_{\frac{k}{n}}\right)
\left(
W_{\frac{k+1}{n}}
-
W_{\frac{k}{n}}
\right),
\qquad n\in\mathbb N,
\]
and define \(\Psi:(-\infty,1)\to\mathbb R_+\) and \(g:[0,1]\to\mathbb R_+\) by
\[
\Psi(z)
=
\frac{-z-\log(1-z)}{2},
\qquad
g(t)
:=
\max_{x\in\mathbb R}
\Psi\!\bigl(u_{xx}(t,x)\bigr).
\]

Then
\[
\lim_{n\to\infty}
\frac{1}{n}
\log
\mathbb E\!\left[e^{\,nS_n}\right]
=
\int_0^1 g(t)\,dt.
\]
\end{theorem}

Although discrete-time hedging in the Brownian setting has been studied extensively, the existing literature has focused primarily on the approximation of hedging strategies and stochastic integrals. The convergence of discrete-time delta hedging on equidistant grids was analyzed in \cite{BKL:00,GT:01,HM:05}, while \cite{G:05,GT:09,GM:12} considered more general deterministic discretization schemes. The papers \cite{F:11,F:14,CRFT:16} further allow for discretization along random stopping times.

These works are concerned with discretization errors under a fixed preference specification. By contrast, our analysis is carried out in a large-deviations regime induced by an exponential scaling of the risk-aversion parameter and is a natural continuation of
the papers \cite{CD,DZ}, which investigated the following question: if an investor is allowed to rebalance only on an equidistant grid of \(n\) trading times, what is the limiting behavior, as \(n\to\infty\), under exponential utility with risk-aversion parameter \(n\)? The corresponding asymptotically optimal rebalancing strategy is characterized by a nonlinear PDE that does not admit an explicit solution.

In the present paper, we analyze the same scaling regime but replace the asymptotically optimal rebalancing strategy with the simpler strategy obtained by directly discretizing the continuous-time delta hedge.

The assumption
\[
\sup_{x\in\mathbb R} f''(x) < 1
\]
is natural for ensuring that the scaling limit is finite, and it was also imposed in \cite{CD,DZ}. To gain some intuition, consider
\[
f(x)=\frac{1}{2}x^2.
\]
Then
\[
S_n=\frac{n}{2}\left(\sum_{k=0}^{n-1}\left(W_{\frac{k+1}{n}}-W_{\frac{k}{n}}\right)^2-1\right),
\]
and therefore
\[
\mathbb{E}\!\left[e^{\,S_n}\right]=\infty
\]
for all \(n\).

\section{Proof of Theorem \ref{thm1}}\label{sec:2}
\begin{proof}
The proof is divided into two steps.

\medskip

\noindent\textbf{Step I.}
In this step, we show the upper bound
$$
\limsup_{n\to\infty}
\frac{1}{n}
\log\!\Bigl(
\mathbb{E}\!\left[e^{\,nS_n}\right]
\Bigr)
\le
\int_0^1 g(t)\,dt .
$$

Choose $n\in\mathbb N$. Clearly,
$
S_n=\sum_{j=0}^{n-1} I_j^n$,
where
$$
I_j^n=
u\!\left(\frac{j+1}{n},W_{\frac{j+1}{n}}\right)
-
u\!\left(\frac{j}{n},W_{\frac{j}{n}}\right)
-
u_x\!\left(\frac{j}{n},W_{\frac{j}{n}}\right)
\left(
W_{\frac{j+1}{n}}-W_{\frac{j}{n}}
\right).
$$

Thus, to complete the proof of the upper bound, it suffices to show that there exists a constant $L>0$ such that, for all $n\in\mathbb N$ and $j=0,1,\dots,n-1$,
\begin{equation}\label{1}
\mathbb E\!\left[e^{\,nI_j^n}\,\middle|\,\mathcal F_{j/n}\right]
 \le e^{\frac{L}{\log n}+g\left(\frac{j}{n}\right)}+\frac{L}{\log n}.
\end{equation}

Indeed, \eqref{1} implies that for every $\epsilon>0$ and all sufficiently large $n$,
$$
\mathbb E\!\left[e^{\,nI_j^n}\,\middle|\,\mathcal F_{j/n}\right]
\le
e^{\epsilon+g\left(\frac{j}{n}\right)},
\qquad j=0,1,\dots,n-1.
$$

Hence, by the tower property of conditional expectation,
\[
\mathbb E\!\left[e^{\,n\sum_{j=0}^{n-1} I_j^n}\right]
\le
e^{\,n\epsilon+\sum_{j=0}^{n-1} g\left(\frac{j}{n}\right)}.
\]

It follows that, since $g$ is continuous,
\[
\limsup_{n\to\infty}
\frac1n
\log
\Bigl(
\mathbb E\!\left[e^{\,n\sum_{j=0}^{n-1} I_j^n}\right]
\Bigr)
\le
\epsilon+\int_{0}^{1} g(t)\,dt.
\]

Since $\epsilon>0$ was arbitrary, we obtain the desired upper bound.

It remains to prove \eqref{1}. Fix $j$ and set
\[
Z^n_j:=\sqrt{n}\Bigl(W_{\frac{j+1}{n}}-W_{\frac{j}{n}}\Bigr).
\]

Next, the Lipschitz continuity of \(f'''\) implies that \(u_t\) and \(u_{xx}\) are uniformly Lipschitz continuous on \([0,1]\times\mathbb{R}\).
Moreover, the constants
\[
a:=\sup_{x\in\mathbb{R}} f''(x),
\qquad
b:=\frac12\sup_{x\in\mathbb{R}} |f''(x)|
\]
satisfy
\[
u_{xx}(t,x)\le a,
\qquad
|u_t(t,x)|\le b,
\]
for all $(t,x)\in[0,1]\times\mathbb{R}$.
Hence, Taylor's theorem yields 
\begin{align*}
n I_j^n
&=
n\left(
u\!\left(\frac{j+1}{n},W_{\frac{j+1}{n}}\right)
-
u\!\left(\frac{j}{n},W_{\frac{j+1}{n}}\right)
\right)
\\
&\quad
+n\Biggl(
u\!\left(\frac{j}{n},W_{\frac{j+1}{n}}\right)
-
u\!\left(\frac{j}{n},W_{\frac{j}{n}}\right)
-
u_x\!\left(\frac{j}{n},W_{\frac{j}{n}}\right)
\left(
W_{\frac{j+1}{n}}-W_{\frac{j}{n}}
\right)
\Biggr)
\\
&\le
u_t\!\left(\frac{j}{n},W_{\frac{j}{n}}\right)
\mathbf 1_{\left\{
|Z^n_j|<\frac{\sqrt n}{\log n}
\right\}}
+
\frac12\,
u_{xx}\!\left(\frac{j}{n},W_{\frac{j}{n}}\right)
|Z^n_j|^2
\mathbf 1_{\left\{
|Z^n_j|<\frac{\sqrt n}{\log n}
\right\}}
\\
&\quad
+\frac{\hat L}{\log n}
+
\left(
b+\frac{a}{2}|Z^n_j|^2
\right)
\mathbf 1_{\left\{
|Z^n_j|>\frac{\sqrt n}{\log n}
\right\}}
\end{align*}
for some constant $\hat L>0$. 

Since \(Z^n_j\sim N(0,1)\) is independent of \(\mathcal F_{j/n}\) and
\(u_t=-\frac12 u_{xx}\), we obtain
\begin{align*}
\mathbb E\!\left[e^{\,nI_j^n}\,\middle|\,\mathcal F_{j/n}\right]
&\le
\mathbb E\!\left[
\exp\!\left(
\frac{\hat L}{\log n}
+
u_t\!\left(\frac{j}{n},W_{\frac{j}{n}}\right)
+
\frac{
u_{xx}\!\left(\frac{j}{n},W_{\frac{j}{n}}\right)
}{2}|Z^n_j|^2
\right)
\right]
\\
&\quad+
\mathbb E\!\left[
\exp\!\left(
b+\frac{a}{2}|Z^n_j|^2
\right)
\mathbf 1_{\left\{
|Z^n_j|>\frac{\sqrt n}{\log n}
\right\}}
\right]
\\
&=
\exp\!\left(
\frac{\hat L}{\log n}
+
\Psi\!\left(
u_{xx}\!\left(\frac{j}{n},W_{\frac{j}{n}}\right)
\right)
\right)
\\
&\quad+
\mathbb E\!\left[
\exp\!\left(
b+\frac{a}{2}|Z^n_j|^2
\right)
\mathbf 1_{\left\{
|Z^n_j|>\frac{\sqrt n}{\log n}
\right\}}
\right].
\end{align*}

The last expectation is bounded by \(\tilde L/\log n\) for some constant
\(\tilde L>0\), by a direct estimate for the standard normal variable \(Z^n_j\).
This yields \eqref{1} and completes the first step.

\medskip
\noindent\textbf{Step II.} We prove the lower bound
\[
\liminf_{n\to\infty}\frac1n\log \mathbb E\!\left[e^{nS_n}\right]
\ge
\int_0^1 g(t)\,dt.
\]

Fix $\epsilon>0$. By the uniform continuity of $u_{xx}$ and $g$, there exists a continuous function $h:[0,1]\to\mathbb R$ such that
\begin{equation}\label{2}
\int_0^1 \Psi(u_{xx}(t,h(t)))\,dt
\ge
\int_0^1 g(t)\,dt-\epsilon.
\end{equation}

Let $\delta>0$ be such that
\begin{equation}\label{3}
|x-y|< \delta
\quad\Longrightarrow\quad
\sup_{0\leq t\leq 1}\Big|
\Psi(u_{xx}(t,x))
-\Psi(u_{xx}(t,y))
\Big|
< \epsilon.
\end{equation}

Fix $n\in\mathbb N$. Let $I_j^n$ and $Z_j^n$
be as in Step I. Define the sets
\[
E_n:=
\left\{
\max_{0\le j\le n}
\Big|
W_{\frac{j}{n}}-h\!\Big(\frac{j}{n}\Big)
\Big|
< \delta
\right\},
\qquad
F_n:=
\left\{
\max_{0\le j<n}|Z^n_j|
\le
\frac{\sqrt n}{\log n}
\right\}.
\]

Using the same arguments as in Step I, Taylor's theorem yields that for any $j=0,1,\dots,n-1$, on the event $E_n\cap F_n$,
\[
nI_j^n
\ge
u_t\!\Big(\frac{j}{n},W_{\frac{j}{n}}\Big)
+\frac12
u_{xx}\!\Big(\frac{j}{n},W_{\frac{j}{n}}\Big)|Z^n_j|^2
-\frac{\hat L}{\log n}.
\]

Hence,
\begin{equation*}
\mathbb{E}\!\left[e^{nS_n}\right]
\ge
e^{-\frac{\hat L n}{\log n}}\,
\mathbb{E}\!\left[
\prod_{j=0}^{n-1}
\exp\!\left(
\frac12\,
u_{xx}\!\left(\frac{j}{n},W_{\frac{j}{n}}\right)
\left(|Z_j^n|^2-1\right)
\right)
\mathbf{1}_{E_n\cap F_n}
\right].
\end{equation*}

Define a probability measure $\mathbb{Q}_n$ by
\[
\frac{d\mathbb{Q}_n}{d\mathbb P}
=
\prod_{j=0}^{n-1}
\exp\!\left(
\frac12\,
u_{xx}\!\left(\frac{j}{n},W_{\frac{j}{n}}\right)
\left(|Z_j^n|^2-1\right)
-
\Psi\!\left(
u_{xx}\!\left(\frac{j}{n},W_{\frac{j}{n}}\right)
\right)
\right).
\]

For a standard normal random variable \(Z \sim N(0,1)\), we have
\[
\mathbb{E}\!\left[
\exp\!\left(
\frac{c}{2}\,(Z^2-1)
\right)
\right]
=
e^{\Psi(c)},
\qquad c<1,
\]
and therefore \(\mathbb{Q}_n\) is indeed a probability measure.
Consequently, from \eqref{3},
\begin{align*}
\mathbb E[e^{nS_n}]
&\ge
e^{-\frac{\hat Ln}{\log n}}
\mathbb E_{\mathbb Q_n}\!\left[
\exp\!\left(
\sum_{j=0}^{n-1}
\Psi\!\left(
u_{xx}\!\left(\frac{j}{n},W_{\frac{j}{n}}\right)
\right)
\right)
\mathbf 1_{E_n\cap F_n}
\right]
\\
&\ge
\exp\!\left(
\sum_{j=1}^n
\Psi\!\left(
u_{xx}\!\left(\frac{j}{n},h\!\left(\frac{j}{n}\right)\right)
\right)
-
n\epsilon
-
\frac{\hat L n}{\log n}
\right)
\mathbb Q_n(E_n\cap F_n).
\end{align*}

Since $\epsilon>0$ was arbitrary, in view of \eqref{2}, it remains to show that
\[
\liminf_{n\to\infty}
\mathbb Q_n(E_n\cap F_n)>0.
\]

To this end, observe that under $\mathbb Q_n$, conditionally on $\mathcal F_{j/n}$,
the random variable $Z^n_j$ is Gaussian with mean $0$ and variance
\[
\frac{1}{1-u_{xx}\!\left(\frac{j}{n},W_{\frac{j}{n}}\right)}.
\]

In particular,
\[
\lim_{n\to\infty}\mathbb Q_n(F_n)=1.
\]
Thus, it remains to show that
\begin{equation}\label{6}
\liminf_{n\to\infty}\mathbb Q_n(E_n)>0.
\end{equation}

For any $n\in\mathbb N$, define the piecewise-constant interpolation
\[
W^n_t:=W_{\lfloor nt\rfloor/n},
\qquad
t\in[0,1].
\]

Since $u_{xx}$ is Lipschitz continuous and
\[
-\infty<\inf_{(t,x)\in [0,1]\times\mathbb R} u_{xx}(t,x)\le \sup_{(t,x)\in [0,1]\times\mathbb R} u_{xx}(t,x)<1,
\]
Theorem 7.4.1 in \cite{EthierKurtz} yields the weak convergence
\[
W^n
\Longrightarrow
X,
\]
where $X$ is the unique strong solution of the SDE
\[
dX_t
=
\frac{dW_t}{\sqrt{1-u_{xx}(t,X_t)}},
\qquad
X_0=0.
\]

Next, $\frac{1}{\sqrt{1-u_{xx}}}$ is bounded, and so the diffusion $X$ has full support (see Chapter VIII in \cite{RY}), and therefore
\[
\mathbb P\!\left(
\sup_{0\le t\le1}|X_t-h(t)|<\delta
\right)
>0.
\]
Together with the relation
\[
\liminf E_n
\supset
\left\{
\sup_{0\le t\le1}|X_t-h(t)|<\delta
\right\},
\]
this yields \eqref{6}.
\end{proof}

\begin{remark}
The proof extends with only minor modifications to the multidimensional setting.
Let $W=(W^1,\ldots,W^d)$ be a standard $d$-dimensional Brownian motion and let
$f:\mathbb R^d\to\mathbb R$ be such that \(D^2f\) is uniformly bounded,
\[
\sup_{x\in\mathbb R^d}
\lambda_{\max}\!\bigl(D^2f(x)\bigr)
<1,
\]
and \(D^3f\) is Lipschitz continuous.

Define
\[
u(t,x):=\mathbb E\bigl[f(x+W_{1-t})\bigr],
\qquad
(t,x)\in[0,1]\times\mathbb R^d.
\]
Then \(u\) solves
\[
u_t+\frac12\Delta u=0,
\qquad
u(1,x)=f(x),
\]
and the Martingale Representation Theorem yields
\[
f(W_1)
=
u(0,0)
+
\int_0^1 \nabla u(t,W_t)\cdot dW_t.
\]

For \(n\in\mathbb N\), set
\[
S_n
:=
f(W_1)-u(0,0)
-
\sum_{k=0}^{n-1}
\nabla u\!\left(\frac{k}{n},W_{\frac{k}{n}}\right)
\cdot
\left(
W_{\frac{k+1}{n}}
-
W_{\frac{k}{n}}
\right).
\]

Then the same argument yields
\[
\lim_{n\to\infty}
\frac1n
\log\mathbb E\!\left[e^{\,nS_n}\right]
=
\int_0^1 g(t)\,dt,
\]
where
\[
g(t)
:=
\sup_{x\in\mathbb R^d}
\Psi\!\bigl(D^2u(t,x)\bigr),
\]
and, for symmetric matrices \(A\) satisfying \(\lambda_{\max}(A)<1\),
\[
\Psi(A)
:=
\frac12
\left(
-\operatorname{tr}(A)-\log\det(I-A)
\right).
\]

Indeed, the only change in the proof is the replacement of the one-dimensional Gaussian identity
\[
\mathbb E\!\left[
e^{\frac12 a(Z^2-1)}
\right]
=
e^{\Psi(a)}
\]
by its multidimensional counterpart
\[
\mathbb E\!\left[
\exp\!\left(
\frac12\bigl(Z^\top A Z-\operatorname{tr}(A)\bigr)
\right)
\right]
=
\exp\!\left(
\frac12
\left(
-\operatorname{tr}(A)-\log\det(I-A)
\right)
\right),
\]
where \(Z\sim N(0,I_d)\).
\end{remark}

\end{document}